\documentclass[]{amsart}

\usepackage{amssymb,stmaryrd,mathrsfs,dsfont}

\theoremstyle{plain}
\newtheorem{theorem}{Theorem}[section]
\newtheorem{lemma}[theorem]{Lemma}

\newtheorem{corollary}[theorem]{Corollary}

\theoremstyle{definition}

\theoremstyle{remark}

\input xy
\xyoption{all}

\def\lcm{\mbox{lcm}}

\begin{document}

\title[The degree of a vertex in the power graph of a finite abelian group]{The degree of a vertex in the power graph of a finite abelian group}

\author[Amit Sehgal]{Amit Sehgal}
\address{Department of Mathematics, Govt. College, Matanhail (Jhajjar), Haryana, India}
\email{amit\_sehgal\_iit@yahoo.com}

\author[Shubh N. Singh]{Shubh N. Singh}
\address{Department of Mathematics, Central University of South Bihar, Gaya, Bihar, India}
\email{shubh@cub.ac.in}


\begin{abstract}
The power graph of a given finite group is a simple undirected graph whose vertex set is the group itself, and there is an edge between any two distinct vertices if one is a power of the other. In this paper, we find a precise formula to count the degree of a vertex in the power graph of a finite abelian group of prime-power order. By using the degree formula, we give a new proof to show that the power graph of a cyclic group of prime-power order is complete. We finally determine the degree of a vertex in the power graph of a finite abelian group.
\end{abstract}

\keywords{Power graphs; Abelian groups; Internal direct products; Linear congruences; Euler's totient function}

\subjclass[2010]{05C20, 05C25, 20K01, 20K25, 11A07}

\maketitle

\section{Introduction}
\vspace{0.1cm}

All groups considered in this paper will be non-trivial finite multiplicative groups. Throughout, $G$ shall denote a group.
By the word \lq graph\rq\  we mean a finite simple undirected graph. The interaction between groups and graphs is still a growing area of research. Kelarev and Quinn \cite{kela00} introduced the \emph{directed power graph} $\overrightarrow{\mathcal{G}}(G)$ of a group $G$ as a digraph whose vertex set is $G$, and there is an edge from vertex $u$ to the other vertex $v$ whenever $v$ is a power of $u$. Further, they gave a technical description of the structure of the directed power graphs of abelian groups. Motivated by this concept, Chakrabarty et al. \cite{chak09} defined the \emph{power graph} $\mathcal{G}(G)$ of a group $G$ as a graph with $G$ as its vertex set, and there is an edge between two distinct vertices if one is a power of the other. Note that the power graph $\mathcal{G}(G)$ is precisely the underlining graph of the directed power graph $\overrightarrow{\mathcal{G}}(G)$. The concept of power graphs of groups has caught the considerable interests of many researchers in the recent years, see for instance \cite{aba13}.

\vspace{0.1cm}
It is easy to see that the power graph of a given group is connected. Chakrabarty et al. \cite{chak09} proved that the power graph $\mathcal{G}(G)$ is complete if and only if $G$ is a cyclic group of prime-power order. They also obtained a formula for the number of edges in the power graph $\mathcal{G}(G)$. Cameron and Ghosh \cite{came11} proved that two abelian groups with isomorphic power graphs are isomorphic. Further, they proved that the only group whose automorphism group is the same as that of its power graph is the Klein $4$-group. Also, Cameron \cite{came10} proved that two groups with isomorphic power graphs have the same number of elements of each order. Curtin and Pourgholi \cite{curt14} proved that among all groups of a given order, the power graph of cyclic group of that order has the maximum size. Many other graph theoretical properties of the power graph of a group are also  investigated, see for instance \cite{came17, chel13, mira12, mogh14, pand18, panda18}. Researchers have also studied the spectra of power graphs of certain groups, see \cite{chapan15, chat17, mehr17, shubh19}.

\vspace{0.1cm}

To analyze the structure of a given graph it is important to look at the degree of each of its vertices. A formula to count the degree of a vertex in the power graph $\mathcal{G}(\mathbb{Z}_n)$ of the cyclic group $\mathbb{Z}_n$  has obtained \cite{doost15}. In this paper, we give an explicit formula to count the degree of a vertex in the power graph of an abelian group. The remainder of the paper is organized as follows. In the next section, we introduce the necessary notations, recall basic concepts of graphs, groups and number theory, and also state some crucial facts. We present our investigation on the degree of a vertex in the power graph of a group in Section 3. Finally, we summarize our investigations in Section 4.

\section{Preliminaries and Notations}
\vspace{0.1cm}

We will assume throughout that the reader is familiar with the basic notions and results concerning group, graph and number theory. The purpose of this section is to introduce relevant concepts and notations used throughout this paper. The reader is referred to \cite{bond08, burt11, foote04} for standard graph, number and group theoretic terminologies, notations and results, respectively.

\vspace{0.1cm}
The \emph{degree} of a vertex $v$ in a graph $\Gamma$ is the number of edges in the graph $\Gamma$ incident with $v$. Let $D$ be a directed graph. The \emph{in-degree} of a vertex $v$ in $D$ is the number of edges with head $v$, and the \emph{out-degree} of $v$ is the number of edges with tail $v$. A \emph{bi-directional edge} in $D$ is a pair of edges which join its two vertices in opposite directions.

\vspace{0.1cm}
Throughout this paper, let $p$ and $n$ denote a prime number and a positive integer, respectively. Let $a$ and $b$ be two positive integers. If $n$ divides $a-b$, we write $a \equiv b \bmod n$ and say that $a$ is congruent to $b$ modulo $n$. The greatest common divisor of $a$ and $b$ is denoted by $(a, b)$, and their least common multiple is denoted by $\lcm(a, b)$. If $(a, b)=1$, we say that $a$ and $b$ are relatively prime. The integers $a$ and $b$ are relatively prime if and only if there exists integers $s, t$ such that $as + bt = 1$. If $c$ is a non-zero integer such that $c$ divides $ab$ and $\gcd(c, a) = 1$, then $c$ divides $b$. Let $\phi(n)$ denote the Euler's totient function of $n$. If $\gcd(a, b) = 1$, then $\phi(ab) = \phi(a)\phi(b)$.

\vspace{0.1cm}
Let $G$ be a group. Let $e$ denote the identity element of $G$ and $|G|$ denote the order of $G$ used throughout the paper. The cyclic group of order $m$ is usually denoted by $\mathbb{Z}_m$. Let $g$ be an arbitrary element of $G$. We denote the order of $g$ by $|g|$, and the cyclic subgroup generated by $g$ by $\langle g \rangle$. Let $H$ and $K$ be two normal subgroups of $G$. If $G$ is the internal direct product of $H$ and $K$, then every element of $G$ can be uniquely expressed as the product of an element of $H$ and an element of $K$. Moreover, every element of $H$ commutes with every element of $K$. If $G$ is an abelian group of order $p_1^{m_1} p_2^{m_2}\cdots p_k^{m_k}$ then, for each $i\;(1\le i \le k)$, the set $G(p_i) = \{x\in G\;|\; x^{p_i^{m_i}} = e\}$ forms a normal subgroup of order $p_i^{m_i}$ in $G$.

\vspace{0.1cm}
We end this section with the following well-known theorem on finite abelian groups.
\begin{theorem}\label{idp-gp}
Let $G$ be an abelian group of order $n$ and let $p_1^{m_1} p_2^{m_2}\cdots p_k^{m_k}$ be the prime decomposition of $n$.
Then $G$ is the internal direct product of the normal subgroups $G(p_i)$ of order $p_i^{m_i}$, that is, \[G \cong G(p_1)\times G(p_2)\times \cdots \times G(p_k).\]
\end{theorem}

\section{Main Results}

\vspace{0.1cm}

In this section, we present our main results on the degree of a group element in the power graph $\mathcal{G}(G)$ of a group $G$. It is easy to see the identity element of $G$ is adjacent to each non-identity group element in $\mathcal{G}(G)$ and so the degree of $e$ in $\mathcal{G}(G)$ is always $|G|-1$.

\vspace{0.1cm}
Let $g$ be a non-identity element of a group $G$. To calculate the degree of $g$ in power graph $\mathcal{G}(G)$, we surely add the in-degree and the out-degree of $g$ and then subtract the number of bi-directional edges incident on $g$ in the directed power graph $\overrightarrow{\mathcal{G}}(G)$.  In order to simplify our notations, we shall write $d_G^+(g)$, $d_G^-(g)$ and $d_G^{\pm}(g)$ respectively to denote the out-degree of $g$, the in-degree of $g$ and the number of bi-directional edges incident on $g$ in the digraph $\overrightarrow{\mathcal{G}}(G)$. By the symbol $\lq d_G(g) \rq$ we mean the degree of $g$ in the power graph $\mathcal{G}(G)$.

\vspace{0.1cm}
By the definition of directed power graph $\overrightarrow{\mathcal{G}}(G)$, it is very easy to check that $d_G^+(g) = |\langle g\rangle| - 1 = |g| - 1$ and $d_G^{\pm}(g) = \phi(|g|)-1$. We thus precisely obtain $d_G(g) = |g| - \phi(|g|) + d_G^-(g)$. To determine the degree $d_G(g)$ of a non-identity group element $g$, it is therefore sufficient to count the in-degree $d_G^-(g)$. However, it is easy see that \[d_G^-(g) = |\{h\in G\;|\; g\neq h \mbox{ and } g  \in \langle h \rangle\}|.\]

We start our investigation on the in-degree $d_G^-(g)$ of a non-identity group element $g$ of $G$. We first present the following elementary lemma for completeness.

\begin{lemma}\label{power-order}
Let $G$ be a group and let $g\in G$ be an element of $G$.
If $|g| = p^m$ and $|g^k|= p^r$, then $k = k'p^{m-r}$ for some integer $k'$ with $(k', p) = 1$.
\end{lemma}

\begin{proof}
We know $|g^k|= \frac{|g|}{(|g|, k)}$ which follows that
\[p^r = \frac{p^m}{(p^m, k)} \Longrightarrow (p^m, k) = p^{m-r} \Longrightarrow k = k'p^{m-r}\]
for some integer $k'$ where obviously $(k', p) = 1$, as required. This completes the proof. 
\end{proof}

\begin{theorem}\label{abe-p-grp}
Let $G = \langle x_1\rangle\times \langle x_2\rangle \times \cdots\times \langle x_n\rangle$ be an abelian $p$-group where $|x_r| = p^{m_r}$ and $1\le m_1 \leq m_2 \leq \cdots \leq m_n$. If $g = \displaystyle\prod_{\alpha=1}^{n} {x_\alpha}^{i_\alpha}$ is a non-identity element of $G$ and $|{x_\alpha}^{i_\alpha}|=p^{t_\alpha}$, then
\[d_G^-(g)=-1+\phi({|g|})\displaystyle  \sum_{\beta=0}^{\min\{m_{k+1}-t_{k+1},\ldots,m_n-t_n\}}p^{\Big(\displaystyle\sum_{j=1}^{n} \min\{m_j, \beta\}\Big)},\]
where $k$ is the smallest non-negative integer such that $|{x_{k+1}}^{i_{k+1}}| \neq 1$.
\end{theorem}

\begin{proof}
Obviously $0\le t_\alpha \leq m_\alpha$ and $|g|=\lcm(p^{t_1},p^{t_2},\ldots,p^{t_n})=p^{t_w}$ where $t_w=\max\{t_1,t_2,\ldots,t_n\}$.
Let $z=\displaystyle\prod_{\alpha=1}^{n} {x_\alpha}^{c_\alpha}$ be an arbitrary element of $G$ such that $z^{\gamma} = g$ for some positive integer $\gamma$ and let $|{x_\alpha}^{c_\alpha}|=p^{s_\alpha}$. Then $|z|=p^{t_w+\beta}$ for some non-negative integer $\beta$. By using Lemma \ref{power-order} we have that $\gamma=\gamma_1 p^{\beta}$ where $\gamma_1$ is a positive integer with $(\gamma_1,p)=1$. Note that the value of $\gamma_1$ uniquely determine the value of $\gamma$ and so the number of possible values of $\gamma_1$ and $\gamma$ are same.


\vspace{0.1cm}
For each $\alpha$, we clearly have that $(x_\alpha^{c_\alpha})^{ \gamma_1 p^{\beta}}={x_\alpha}^{i_\alpha}$ and so
\[|(x_\alpha^{c_\alpha})^{ \gamma_1 p^{\beta}}|=|{x_\alpha}^{i_\alpha}|\Longrightarrow \frac{p^{s_\alpha}}{(p^{s_\alpha},\gamma_1 p^{\beta})}=p^{t_\alpha} \Longrightarrow s_\alpha=t_\alpha+\min\{s_\alpha, \beta\}.\]

In particular, $s_{k+1} = t_{k+1} + \min\{s_{k+1}, \beta\}$ where $k$ is the smallest non-negative integer such that $t_{k+1} \neq 0$. This obviously gives $\beta \le m_{k+1} - t_{k+1}$. As $1 \leq m_1 \leq \cdots \leq m_n$, it is easy to check that $0 \le \beta \leq \min\{m_{k+1}-t_{k+1},\ldots,m_n-t_n\}$.


\vspace{0.1cm}
We now consider, for an arbitrary fixed $\alpha$, the following two cases.

\vspace{0.1cm}
\textbf{Case 1.} If $t_\alpha=0$, then $i_\alpha \equiv 0 \bmod \; p^{m_\alpha}$.
Since $z^{\gamma_1 p^{\beta}}= g$ with $(\gamma_1, p)=1$, we have that \[{x_\alpha}^{c_\alpha \gamma_1 p^{\beta}}={x_\alpha}^{{i_\alpha}} \Longrightarrow c_\alpha \gamma_1 p^{\beta} \equiv i_\alpha \bmod p^{m_\alpha} \Longrightarrow c_\alpha p^{\beta} \equiv 0 \bmod p^{m_\alpha}.\]

If $\beta \ge m_\alpha$, then $c_\alpha p^{\beta} \equiv 0 \bmod p^{m_\alpha}$ for all $c_\alpha$ and consequently the total number of possible ways of choosing $c_\alpha$ in a complete residue system modulo $p^{m_\alpha}$ is $p^{m_\alpha}$.

\vspace{0.1cm}
If $\beta < m_\alpha$, then $c_\alpha$ is divisible by $p^{m_\alpha-\beta}$.
We know that the number of integers divisible by $p^{m_\alpha-\beta}$ in a complete residue system modulo $p^{m_\alpha}$ is $p^{\beta}$ and consequently the total number of possible ways of choosing $c_\alpha$ is $p^{\beta}$.

\vspace{0.1cm}

Thus, in this case, the total number of possible ways of choosing $c_\alpha$ is $p^{\min\{m_\alpha,\beta\}}$.

\vspace{0.1cm}
\textbf{Case 2.} If $t_\alpha \neq 0$, recall that $s_\alpha=t_\alpha+\min\{s_\alpha, \beta\}$ which certainly gives  $s_\alpha=t_\alpha+\beta$ and so $\beta \le m_\alpha - t_\alpha$. Given that $|{x_\alpha}^{c_\alpha}| = p^{s_\alpha}$, by using Lemma \ref{power-order} we obtain $c_\alpha=d_\alpha p^{m_\alpha-t_\alpha-\beta}$ where $d_\alpha$ is a positive integer with $(d_\alpha,p)=1$. Also  $|{x_\alpha}^{i_\alpha}|= p^{t_\alpha}$, by using Lemma \ref{power-order} we obtain
$i_\alpha=j_\alpha p^{m_\alpha-t_\alpha}$ where $j_\alpha$ is a positive integer with $(j_\alpha,p)=1$. As ${x_\alpha}^{c_\alpha \gamma_1 p^{\beta}}={x_\alpha}^{i_\alpha}$ which follows that \[{x_\alpha}^{(d_\alpha p^{m_\alpha-t_\alpha-\beta}) \gamma_1 p^{\beta}}={x_\alpha}^{j_\alpha p^{m_\alpha-t_\alpha}}\Longrightarrow d_\alpha \gamma_1\equiv j_\alpha \bmod p^{t_\alpha}.\]
and subsequently $d_w \gamma_1\equiv j_w \bmod p^{t_w}$. As $(\gamma_1, p) = 1$, the total number of possible values of $\gamma_1$ is $\phi(p^{t_w})=\phi(|g|)$.

\vspace{0.1cm}
Recall that $(\gamma_1, p)=1$, there exists an integer $\delta_1$ with $(\delta_1, p) = 1$ such that $\gamma_1 \delta_1 \equiv 1 \bmod p^{t_\alpha}$.
Since $j_\alpha \equiv d_\alpha \gamma_1 \bmod p^{t_\alpha}$, multiplying both sides by $\delta_1$ we get
\[j_\alpha \delta_1\equiv d_\alpha \gamma_1 \delta_1\bmod p^{t_\alpha}\Longrightarrow j_\alpha \delta_1\equiv d_\alpha \bmod\; p^{t_\alpha}.\]

Now $c_\alpha = d_\alpha p^{m_\alpha-t_\alpha-\beta}$ where $(d_\alpha, p)=1$, choose $d_\alpha$ from a reduced residue system modulo $p^{t_\alpha+\beta}$ such that $d_\alpha \equiv \delta_1 j_\alpha \bmod p^{t_\alpha}$. Then the total number of possible values of $c_\alpha$ is $\frac{\phi(p^{t_\alpha+\beta})}{\phi(p^{t_\alpha})}=p^{\beta}$.

\vspace{0.1cm}
Hence, the total number of possible values of $c_\alpha$ is $p^{\min\{m_\alpha,\beta\}}$ or $p^{\beta}$ accordingly $t_\alpha = 0$ or $t_\alpha\neq 0$.
Recall that $0 \le \beta \le\min\{m_{k+1}-t_{k+1},\ldots,m_n-t_n\})$ and the total number of possible values of $\gamma$ is $\phi(|g|)$. Since we are not taking $z = g$ into consideration, we consequently have that
\[d_G^-(g)=-1 + \phi({|g|})\displaystyle  \sum_{\beta=0}^{\min\{m_{k+1}-t_{k+1},\ldots,m_n-t_n\}}p^{\Big(\displaystyle\sum_{j=1}^{n} \min\{m_j,\beta\}\Big)}.\] 
\end{proof}

If $G$ is a cyclic group of prime-power order, the following corollary of Theorem \ref{abe-p-grp} shows that the power graph $\mathcal{G}(G)$ is complete.

\begin{corollary}
If $G$ is a cyclic $p$-group, then the power graph $\mathcal{G}(G)$ is complete.
\end{corollary}

\begin{proof}
Assume $|G| = p^m$ for some positive integer $m$. Clearly the degree of identity element in $\mathcal{G}(G)$ is $p^m - 1$. Let $g$ be an arbitrary non-identity element of $G$. Then $|g| = p^k$ for some positive integer $k$ with $1\le k \le m$. By using Theorem \ref{abe-p-grp}, we obtain $d_G^-(g)= -1 + \phi(p^{k})\displaystyle\sum_{\beta=0}^{m-k} p^{\beta} = p^m - p^{k-1} -1$. Thus \[d_G(g) = |g| - \phi(|g|) + d_G^-(g) = p^k - (p^k - p^{k-1}) + p^m - p^{k-1} - 1 = p^m - 1,\] as required. This completes the proof. 
\end{proof}

If $G$ is the internal direct product of its normal subgroups $H$ and $K$ whose orders are relatively prime, then we obtain the following crucial Theorem.

\begin{theorem}\label{co-groups}
Let $G$ be a group and let $H$ and $K$ be two normal subgroups of $G$ such that $(|H|,|K|) = 1$. If $G$ is the internal direct product of the subgroups $H$ and $K$, then for an element $z = xy$ of the group $G$ where $x\in H$ and $y \in K$,
\begin{enumerate}
\item[\rm(i)] $d_G^+(z) = \Big(d_H^+(x)+1\Big)\Big(d_K^+(y)+1\Big)-1$.
\item[\rm(ii)] $d_G^{\pm}(z) = \Big(d_H^{\pm}(x)+1\Big)\Big(d_K^{\pm}(y)+1\Big)-1$.
\item[\rm(iii)] $d_G^-(z) = \Big(d_H^-(x)+1\Big)\Big(d_K^-(y)+1\Big)-1$.
\end{enumerate}
\end{theorem}

\begin{proof}
Given $(|H|,|K|) = 1$, it is easy to check that $( |x|,|y| ) = 1$ and so $|z|=\lcm(|x|, |y|) = |x| |y|$.
\begin{enumerate}
\item[\rm(i)] Clearly $d_H^+(x) = |x|-1$ and $d_K^+(y) = |y|-1$. Hence \[d_G^+(z) = |z|-1 = |x| |y|-1 = (d_H^+(x)+1)(d_K^+(y)+1)-1.\]

\item[\rm(ii)] Clearly $d_H^{\pm}(x)=\phi(|x|)-1$ and  $d_K^{\pm}(y)=\phi(|y|)-1$. Since $(|x|, |y|) = 1$, exploiting multiplicative nature
of the Euler's totient function, we obtain \[d_G^{\pm}(z) = \phi(|x||y|) - 1= \phi(|x|) \phi(|y|) - 1 = (d_H^{\pm}(x) +1) (d_K^{\pm}(y) +1) - 1.\]

\item[\rm(iii)] Let $z_1 \in G$ be an arbitrary element such that $z\neq z_1$ and $z\in \langle z_1\rangle$. Then $z_1 = x_1 y_1$ where $x_1 \in H$ and $y_1 \in K$. Since $z\in \langle z_1\rangle$, it follows that $z = z_1^k$ for some positive integer $k$.
This gives $xy = (x_1 y_1)^k$. From the internal direct product property we get
$xy = x_1^k y_1^k \Longrightarrow x_1^{-k} x =  y_1^k y^{-1} \Longrightarrow x_1^k  = x\;\;\mbox{ and }\;\; y_1^k = y$. Since we consider the cases when $x_1 = x$ or $y_1 = y$, but not the case when $x_1 = x$ and $y_1 = y$ simultaneously, it follows that \[d_G^-(z) \leq (d_H^-(x)+1)(d_K^-(y)+1)-1.\]

Conversely, let $x \in \langle x_1 \rangle$ and $y \in \langle y_1 \rangle$ where $x_1 \in H$  and $y_1 \in K$. Then there exist positive integers $s$ and $t$  such that $x_1^s = x$ and $y_1^t = y$. Given that $(|H|,|K|)=1$, there exist integers $m$ and $r$ such that $m|H|+r|K|=1$. We now consider the integer
\begin{equation*}
\begin{split}
s-t & = s(m|H| + r|K|)-t (m|H| + r|K|)\\
& = m s|H| + r s|K|- m t|H| - r t|K|\\
& =  r |K| (s - t) - m |H| (t- s)\\
& = \beta - \alpha,
\end{split}
\end{equation*}
where $\alpha = m|H|(t-s)$ and $\beta  = r |K| (s-t)$. It follows that $s + \alpha  = t + \beta$. Observe that
 $x_1^{s + \alpha} = x_1^s x_1^{\alpha} = x_1^s x_1^{m|H|(t -s)} = x_1^s (x_1^{|H|})^{m(t-s)} = x_1^s = x$, and
$y_1^{t + \beta} =y_1^t y_1^{\beta}= y_1^t y_1^{r |K| (s-t)} = y_1^t (y_1^{|K|})^{r(s-t)} = y_1^t = y$. From the commutative property of elements in the internal direct product of $H$ and $K$, we thus obtain
\[(x_1 y_1)^{s+\alpha}=x_1^{s+\alpha} y_1^{s+\alpha}= x_1^{s+\alpha} y_1^{t+\beta} = xy,\]
as $s+\alpha = t + \beta$. We here consider the cases when $x_1 = x$ or $y_1 = y$, but not the case when $x_1 = x$ and $y_1 = y$ simultaneously, it certainly follows that  $d_G^-(z) \geq (d_H^-(x)+1)(d_K^-(y)+1)-1$.
Combining above obtained two inequalities, we finally have that \[d_G^-(z) = (d_H^-(x)+1)(d_K^-(y)+1)-1.\] 
\end{enumerate}
\end{proof}

As a consequence of Theorem \ref{co-groups} and an induction argument on the number of subgroups, we obtain the following extension of Theorem \ref{co-groups}.
\begin{theorem}\label{co-groups1}
Let $G$ be a group and let $H_1, H_2, \ldots, H_n$ be  normal subgroups of $G$ such that $(|H_i|,|H_j|) = 1$ when $i\neq j$. If $G$ is the internal direct product of subgroups $H_1, H_2, \ldots, H_n$, then for an element $x = x_1 x_2\ldots x_n$ of the group $G$ where $x_i\in H_i$,
\begin{enumerate}
\item[\rm(i)] $d_G^+(x)=\Big(\displaystyle\prod_{i=1}^{n}\big(d_{H_i}^+(x_i)+1\big)\Big)-1$.
\item[\rm(ii)]$d_G^{\pm}(x)=\Big(\displaystyle\prod_{i=1}^{n}\big(d_{H_i}^{\pm}(x_i)+1\big)\Big)-1$.
\item[\rm(iii)]$d_G^-(x)=\Big(\displaystyle\prod_{i=1}^{n}\big(d_{H_i}^-(x_i)+1\big)\Big)-1$.
\end{enumerate}
\end{theorem}

Suppose $G$ is an abelian group of order $p_1^{m_1} p_2^{m_2}\ldots p_k^{m_k}$, where $p_i$'s are distinct primes. Then obviously the orders of normal subgroups $G(p_i)$ and $G(p_j)$ of group $G$ are relatively prime when $i\neq j$. As a direct consequence of Theorem \ref{idp-gp} and Theorem \ref{co-groups1}, we thus obtain the following straightforward Theorem \ref{deg-abe}.

\begin{theorem}\label{deg-abe}
Let $G$ be an abelian group of order $n$ and let $p_1^{m_1} p_2^{m_2}\ldots p_k^{m_k}$ be the prime decomposition of $n$. Then for a non-identity element $x = x_1 x_2 \ldots x_k$ of the abelian group $G$ where $x_i \in G(p_i)$,
\begin{enumerate}
\item[\rm(i)] $d_G^+(x)=\Big(\displaystyle\prod_{i=1}^{k}\big(d_{G(p_i)}^+(x_i)+1\big)\Big)-1$.
\item[\rm(ii)]$d_G^{\pm}(x)=\Big(\displaystyle\prod_{i=1}^{k}\big(d_{G(p_i)}^{\pm}(x_i)+1\big)\Big)-1$.
\item[\rm(iii)]$d_G^-(x)=\Big(\displaystyle\prod_{i=1}^{k}\big(d_{G(p_i)}^-(x_i)+1\big)\Big)-1$.
\end{enumerate}
\end{theorem}

\section{Conclusion}
The results reported in the present paper show that the indegree of a group element is only important component in determining the degree of the group element in the power graph. We have obtained a precise formula to count the degree of a vertex in the power graph of an abelian group of prime-power order. By using the degree formula, we presented an alternative proof to show that the power graph of a cyclic group of prime-power order is complete. By using the fundamental theorem of finite abelian groups, we finally determined the degree of a vertex in the power graph of a finite abelian group.

\end{document}